\documentclass[preprint,12pt]{elsarticle}

\usepackage{amsmath,amsfonts,amsbsy,amstext,amsopn,amssymb}
\usepackage{longtable}
\usepackage{theorem}

{\theorembodyfont{\slshape}
\newtheorem{theorem}{Theorem}

\newtheorem{lemma}{Lemma}

} {\theorembodyfont{\rmfamily}
\newtheorem{definition}{Definition}

\newtheorem{example}{Example}

}

\usepackage{latexsym}
\usepackage{amssymb}
\usepackage{graphicx}
\usepackage{algorithm}

\newcommand {\eq} [1] {\begin{equation}\label{#1}}
\newcommand {\en} {\end{equation}}
\newcommand {\R}        {{\mathbb R}}
\newcommand {\C}        {{\mathbb C}}

\newcommand {\Rnn}      {\R^{n \times n}}

\newcommand {\mat}      [1] {\left[\begin{array}{#1}}
\newcommand {\rix}          {\end{array}\right]}

\newcommand {\diag}     {\mathop{\rm diag}\nolimits}

 \catcode`@=11
 \font\tenex=cmex10 % math extension
 \newdimen\p@renwd
 \setbox0=\hbox{\tenex B} \p@renwd=\wd0 % width of the big left [
 \def\bmat#1{\begingroup \m@th
   \setbox\z@\vbox{\def\cr{\crcr\noalign{\kern2\p@\global\let\cr\endline}}%
     \ialign{$##$\hfil\kern2\p@\kern\p@renwd&\thinspace\hfil$##$\hfil
       &&\quad\hfil$##$\hfil\crcr
       \omit\strut\hfil\crcr\noalign{\kern-\baselineskip}%
       #1\crcr\omit\strut\cr}}%
   \setbox\tw@\vbox{\unvcopy\z@\global\setbox\@ne\lastbox}%
   \setbox\tw@\hbox{\unhbox\@ne\unskip\global\setbox\@ne\lastbox}%
   \setbox\tw@\hbox{$\kern\wd\@ne\kern-\p@renwd\left[\kern-\wd\@ne
     \global\setbox\@ne\vbox{\box\@ne\kern2\p@}%
     \vcenter{\kern-\ht\@ne\unvbox\z@\kern-\baselineskip}\,\right]$}%
   \null\;\vbox{\kern\ht\@ne\box\tw@}\endgroup}
 \catcode`@=12
\newcommand{\vect}{{\sf{vec}}}

\newcommand{\rt}{{\top }}

\def\e{\varepsilon}
\def\Oh{{\mathcal O}}
\def\Prob{\mathop{\rm Prob}}

\def\udots{\mathinner{\mkern1mu\raise-1pt\vbox{\kern7pt\hbox{.}}\mkern2mu
    \raise2pt\hbox{.}\mkern2mu\raise5pt\hbox{.}\mkern1mu}}

\usepackage{lineno,hyperref}
\modulolinenumbers[5]

\journal{Journal of \LaTeX\ Templates}

\begin{document}

\begin{frontmatter}

\title{Backward errors and small sample condition  estimation
for $\star$-Sylveter equations}
%\tnotetext[mytitlenote]{Fully documented templates are available in the elsarticle package on \href{http://www.ctan.org/tex-archive/macros/latex/contrib/elsarticle}{CTAN}.}

%% Group authors per affiliation:
\author[mymainaddress]{Huai-An Diao\corref{mycorrespondingauthor}
}
\cortext[mycorrespondingauthor]{Corresponding author}
%\address{School of Mathematics and Statistics, Northeast
%Normal University, No. 5268 Renmin Street, Chang Chun 130024, P.R.
%of China.}
\fntext[myfootnote]{Email address: hadiao@nenu.edu.cn, hadiao78@yahoo.com (H. A. Diao), yhhy8899@163.com (H. Yan), eric.chu@monash.edu (E. K. Chu)}

%% or include affiliations in footnotes:
\author[mymainaddress,mysecondaryaddress]{Hong Yan}
%\ead[url]{www.elsevier.com}

\author[mythirdaddress]{Eric King-wah Chu}

%\ead{hadiao@nenu.edu.cn, hadiao78@yahoo.com (H. A. Diao), yhhy8899@163.com (H. Yan), eric.chu@monash.edu (E. K. Chu)}

\address[mymainaddress]{School of Mathematics and Statistics, Northeast
Normal University, No. 5268 Renmin Street, Chang Chun 130024, P.R.
of China.}
\address[mysecondaryaddress]{Current address: Liewu Middle school, No. 8 Hangkong Road, Chengdu  610000, P.R. of China.}
\address[mythirdaddress]{School of Mathematical Sciences, Monash University, 9 Rainforest Walk, Victoria 3800, Australia.}

\begin{abstract}
In this paper, we adopt a componentwise perturbation analysis for $\star$-Sylvester equations. Based on the small condition estimation (SCE), we devise the algorithms to estimate normwise, mixed and componentwise  condition numbers for  $\star$-Sylvester equations.  We also define a componentwise backward error with a sharp and easily computable bound.  Numerical examples illustrate that our algorithm under componentwise perturbations produces reliable estimates,  and the new derived computable bound for the componentwise backward error is sharp and reliable for well conditioned and moderate ill-conditioned $\star$-Sylvester equations under large or small perturbations.
\end{abstract}

\begin{keyword}
$\star$-Sylvester equation \sep condition number \sep componentwise
perturbation \sep backward error \sep small-sample condition estimation
\MSC[2010] 15A09 \sep 15A12 \sep 65F35
\end{keyword}

\end{frontmatter}

%\linenumbers

\section{Introduction}
Consider the $\star$-Sylvester equation:
\begin{equation}
\label{sec1:mat1 eq}
AX\pm X^\star B^\star=C,
\quad A,B,C\in\C^{n\times n},
\end{equation}
where $\star$ denote  the conjugate transpose of a complex matrix. The $\star$-Sylvester equation  arises in the perturbation of palindromic eigenvalue problems and the solution of the $\star$-Ricati equation \cite{ChiangChuLin2012AMC}:
$$
XAX^\star+XB+CX^\star+D=0,\quad A,B,C,D\in\C^{n\times n}.
$$

The following lemma gives the sufficient and necessary condition for the existence and uniqueness of the solution of ~$\star$-Sylvester equation, which appeared in \cite{ChiangChuLin2012AMC}. The solvability of \eqref{sec1:mat1 eq} was also investigated in \cite{DeTeranDopico2011ELA}. Note that the spectrum $\sigma(A,B)$ contains the ordered
pairs $(a_i,\, b_i)$ and represents the generalized eigenvalues $\lambda_i=a_i/b_i$ of the matrix pencil $A-\lambda B$ or matrix pairs $(A,\,B)$.
\begin{lemma}\label{lem:chu}\cite{ChiangChuLin2012AMC}
For the $\star$-Sylvester equation
 \begin{equation*}
AX\pm X^\star B^\star=C,
\end{equation*}
where~$A,B,C\in\C^{n\times n} $,  the solution exists and is unique if and only if, for $\{(a_{ii},\,b_{ii})\}=\sigma(A,\,B)$, the following conditions are satisfied:
$$
a_{ii}a_{jj}^\star-b_{ii}b_{jj}^\star\neq 0,\,\quad (\forall i\neq j);
$$
and, for $\lambda_i=a_{ii}/b_{ii}$ and all $i$,
$$
a_{ii}\pm b_{ii}\neq 0 \quad (\mbox{for}\quad \star=\top),\quad |\lambda_i|\neq 1\quad (\mbox{for} \quad \star=H).
$$
%$$
%\ \lambda \in \sigma(A,B)\Rightarrow\lambda^{-*}\notin \sigma(A,B)
%$$
%where $\lambda\in \R$, ~$\lambda^{-*} = 1/\lambda $; and $\lambda\in \C$, $\lambda^{-*} = 1/\overline{\lambda}$.
\end{lemma}

In this paper, we  focus on the case where $A,B,C$ are real with the plus sign in \eqref{sec1:mat1 eq}, i.e., ~the $\rt$-Sylvester equation
%\begin{equation*}
%%\label{sec2:mat eq}
%AX\pm X^\rt B^\rt=C,
%\end{equation*}
%where $A,B,C\in\R^{n\times n}$. Similar results in this paper can be derived for \eqref{sec1:mat1 eq}. For the above~$\rt$-Sylvester equation, we only consider the case of the sign being  \lq$+$\rq, i.e.,
\begin{equation}
\label{sec2:mat eq}
AX+X^\rt B^\rt=C.
\end{equation}
%and for the case of the sign \lq$-$\rq, we can use~$-B$ to replace $B$.
Similar results can easily be developed for the general cases and will be ignored.

For $A\in\R^{m\times n}$, $\vect(A)$ stacks the columns of $A$ to a vector. The Kronecker Product  $A\otimes B=(a_{ij}B)\in \R^{mp \times nq}$ for $A=(a_{ij}) \in \R^{m\times n}$ and $B\in \R^{p\times q}$ and \eqref{sec2:mat eq} is equivalent to
\begin{equation}
\label{sec2:mat eqqq}
P~\vect(X)=\vect(C),
\end{equation}
where $P=I\otimes A+(B\otimes I)\Pi$ and $\Pi$ is the permutation matrix satisfying $\vect(A^\rt)\Pi=\vect(A)$. Under the conditions in Lemma \ref{lem:chu}, the coefficient matrix in \eqref{sec2:mat eqqq} is invertible.

An extension of the Bartels-Stewart algorithm \cite{BartelsStewart1972} to uniquely solvable $\rt$-Sylvester equations was presented in \cite{ChiangChuLin2012AMC}; see also \cite{DeTeranDopico2011ELA,VorontsovIkramov2011}. The algorithm first computes a generalized real Schur decomposition \cite{GolubVanLoan1996} of $A-\lambda B^\rt$:
\begin{equation}\label{eq:generilized schur}
A=WT_AV^\rt, \quad B^\rt =WT_B V^\rt,
\end{equation}
where $T_A \in \R^{n\times n}$ is upper quasi-triangular, $T_B$ is upper triangular and $V,\, W\in \R^{n\times n}$ are orthogonal. Defining $Y=V^\rt X W$, the factorization in \eqref{eq:generilized schur} allow us to transform \eqref{sec2:mat eq} to the equivalent $\rt$-Sylvester equation
$$
T_AY+Y^\rt T_B^\rt=W^\rt C W.
$$
The (block) triangular structures of $T_A$ and $T_B$ yield $Y$ by a simple substitution procedure and the algorithm is completed by the retrieval of $X=VYW^\rt$. The total flop count of the algorithm is of $\Oh(67\frac{1}{6}n^3)$; see \cite{ChiangChuLin2012AMC} for details.

In sensitivity analysis, condition numbers are important, measuring the {\em worst-case} effect of small changes
in the data on the solution. A problem with a large condition number is called {\em ill-posed} \cite{Higham2002Book}, and the computed solution to the problem via any numerical algorithms cannot be reliable. For the perturbation analysis of Layapunov, (generalized) Sylvester and Ricatti equations, the
readers are referred to~\cite{GahinetLaubKenneyHewer1990IEEE,GhavimiLaub1995ResidualIEEE,GhavimiLaub1995,HewerKenney1988Lyapunov,Kagstrom1994SIMAX,Higham1993BITSyl} and references therein. Componentwise perturbation analysis can give sharper error bounds than those based on normwise perturbation analysis because it can better capture the condition of the problem with respect to the scaling and sparsity of the data; see the comprehensive review \cite{Higham1994SurvyComp}. Diao {\it {\it et al}.} \cite{DiaoXiangWei2012NLAA} introduced componentwise perturbation analysis for Sylvester equation. The explicit expressions for normwise, mixed and componentwise condition numbers  were derived. For the perturbation analysis for $\star$-Sylvester equation \eqref{sec1:mat1 eq}, Chiang {\it et al}. \cite{ChiangChuLin2012AMC} studied the normwise perturbation analysis, both the normwise perturbation error bounds and normwise backward errors were investigated. Assume that there are perturbations $\Delta A,\, \Delta B$ and $\Delta C$ on $A,\, B$ and $C$ respectively, and when the norms of perturbation matrices are sufficiently small, the following perturbed $\rt$-Sylvester equation
\begin{equation}\label{eq:perturbed T-Syl}
(A+\Delta A)(X+\Delta X)+(X+\Delta X)^\rt (B+\Delta B)^\rt=C+\Delta C
\end{equation}
has the unique solution $X+\Delta X$. The normwise perturbation bound for $X$ is given by \cite[Sec. 2.2.3]{ChiangChuLin2012AMC}
$$
\frac{\|\Delta X\|_F}{\|X\|_F}\leq \frac{\kappa(P)}{1-\kappa(P)\|\Delta P\|_F/\|P\|_F}\left(\frac{\|\Delta C\|_F}{\|C\|_F}+\frac{\|\Delta P\|_F}{\|P\|_F}\right),
$$
where $\|A\|_F$ is Frobenius norm of $A$, $\kappa(P)=\|P\|_F\|P^{-1}\|_F$ and $\|\Delta P\|_F=\|\Delta A\|_F+\|\Delta B\|_F$. The normwise backward error for the computed solution $Y$ of \eqref{sec2:mat eq} is defined as
\begin{align*}
 \eta(Y)=\min ~\{~\epsilon :& (A+\Delta A)Y+Y^\rt (B+\Delta B)^\rt =C+\Delta C
 \\&\|\Delta A\|_F \leq \epsilon\|A\|_F,\,\|\Delta B\|_F\leq \epsilon\|B\|_F,\,
\|\Delta C\|_F\leq \epsilon\|C\|_F\}.
\end{align*}
The upper bound for $\eta(Y)$ is given in \cite[Sec. 2.2.2]{ChiangChuLin2012AMC} as
\begin{equation}\label{eq:eta}
\eta(Y) \leq \frac{(\|A\|_F+\|B\|_F)\|Y\|_F+\|C\|_F}{\left[(\|A\|_F^2+\|B\|_F^2)\|X^{-1}\|_2^{-2}+\|C\|_F^2\right]^{1/2}}\cdot \frac{\|R\|_F}{(\|A|_F+\|B\|_F)\|Y\|_F+\|C\|_F},
\end{equation}
where $R=C-AY-Y^\rt B^\rt$. Recently, Yan \cite{Yan2015} introduced componentwise perturbation analysis for $\star$-Sylvester  equation,  defined and obtained normwise, mixed and compoentwise condition numbers for $\rt$-Sylvester  equation as follows
{\small
\begin{align}
\kappa^\mathrm{T-SYL} &= \lim_{\epsilon_1 \rightarrow 0}\sup_{\mbox{\tiny $\begin{array}{c}
\|\Delta A\| \leqslant \epsilon_1 \|A\|_{F} \\
\|\Delta B\| \leqslant \epsilon_1 \|B\|_{F} \\
\|\Delta C\| \leqslant \epsilon_1 \|C\|_{F}
\end{array} $}}  \frac{\|\Delta X\|_{F}}{\epsilon_1\|X\|_{F}} = \frac{ \left\| P^{-1}
[X^\rt\otimes I,(I\otimes X^T)\Pi,-I] \right\|_F
\left\| \begin{matrix}\vect(|A|)\cr \vect(|B|)\cr
\vect(|C|)\end{matrix}\right\|_F }
{ \|X\|_{F}}  ,\label{eq:cond definition}\\
m^\mathrm{T-SYL}&= \lim_{\epsilon \rightarrow
0}\sup_{\mbox{\tiny $\begin{array}{c}
|\Delta A| \leqslant \epsilon |A| \\
|\Delta B|  \leqslant \epsilon |B| \\ |\Delta C| \leqslant \epsilon
|C|
\end{array} $}} \frac{\|\Delta X\|_{\max}}{\epsilon  \|X\|_{\max} }\nonumber \\
&=
 \frac{   \left\| ~\left|P^{-1}
(X^\rt\otimes I)\right| \vect(|A|) +\left|P^{-1} ((I\otimes X^T)\Pi)
\right|\vect(|B|)+\left|P^{-1} \right|\vect(|C|) ~
\right\|_\infty   }{ \|X\|_{\max}},\nonumber\\
c^\mathrm{T-SYL} &= \lim_{\epsilon \rightarrow 0}\sup_{\mbox{\tiny $\begin{array}{c} |\Delta A| \leqslant \epsilon |A| \\
|\Delta B|  \leqslant \epsilon |B| \\ |\Delta C| \leqslant \epsilon
|C|
\end{array} $}} \frac{1}{\epsilon } \left\| \frac{\Delta X}{ X }
\right\|_{\max}\nonumber \\
& =  \left\| \frac{ ~\left|P^{-1}
(X^\rt\otimes I)\right| \vect(|A|) +\left|P^{-1} ((I\otimes X^T)\Pi)
\right|\vect(|B|)+\left|P^{-1} \right|\vect(|C|)
}{\vect(|X|)}\right\|_\infty,\nonumber
\end{align}
}
where  $\|A\|_\infty$ is $\infty$ norm,  $\|A\|_{\max}=\max_{i,j}|a_{ij}|$, $|\Delta A|\leq \epsilon |A|$ is intepreted componentwisely, $\Delta X/X$ is the componentwise quotient (when a denominator is zero, the corresponding numerator must be zero and the corresponding ratio is defined as zero), $I$ is identity matrix and
$$
\epsilon_1 =  \max \left \{ \frac{\|\Delta A\|_F}{ \|A\|_F },
\frac{\|\Delta B\|_F}{ \|B\|_F }, \frac{\|\Delta C\|_F}{ \|C\|_F }
\right \}.
$$

The normwise, mixed and componentwise condition numbers  were also studied in \cite{GengWang2015}.  Explicit expressions have been derived without the corresponding reliable and efficient estimation. In this paper, we introduce the SCE-based condition estimation for the $\rt$-Sylvester  equation, as well as the associated componentwise backward error.

The following example from \cite[Sec. 6]{Yan2015} shows that there are big differences between $m^\mathrm{T-SYL}$, $c^\mathrm{T-SYL}$ and $\kappa^\mathrm{T-SYL}$, illustrating that the mixed and componentwise condition number better capture the condition of $\rt$-Sylvester  equation with respect to the scalling and sparsity of the input data.
\begin{example}
Let $0< \e <1$, for the following $\rt$-Sylvester  equation
$$
\begin{bmatrix}1&0\cr 0&\e\end{bmatrix}X+X^\rt \begin{bmatrix}1&0\cr
0&0\end{bmatrix}=\begin{bmatrix}2&0\cr 0&\e\end{bmatrix},
$$
it is easy to see that $X=I_2$ is the exact solution. We have $m^\mathrm{T-SYL}=c^\mathrm{T-SYL}=2$ and $\kappa^\mathrm{T-SYL}=\sqrt{\frac{63}{4}+\frac{15}{8}\varepsilon^2+\frac{27}{\varepsilon^2}}=
\Oh(\frac{1}{\varepsilon})$  from \eqref{eq:cond definition}.
\end{example}
From the above example, we see that that the perturbation bounds based on normwise condition number may severely overestimate errors. Another issue is that the expreesions for $\kappa^\mathrm{T-SYL},\, m^\mathrm{T-SYL} $ and $c^\mathrm{T-SYL} $ involve the Kronecker product, which involves higher
dimensions and prevents the efficient
estimation of the condition numbers.

In practice, the
problem of how condition numbers are estimated efficiently is
critical \cite[Chapter 15]{Higham2002Book}. Kenny and
Laub~\cite{KenneyLaub_SISC94} developed the method of the small-sample statistics condition estimation (SCE), applicable for general
matrix functions, linear equations~\cite{KenneyLaubReese1998Linear}, eigenvalue
problems~\cite{GudmundssonKenneyLaub1997Eig}, linear least squares problem \cite{KenneyLaubReese1998LS}
and roots of polynomials~\cite{LaubXia2008Root}. Recently, SCE had been used to estimate the condition of Sylvester  equations \cite{DiaoShiWei2013,DiaoXiangWei2012NLAA}. In this paper we  devise
SCE algorithms to estimate the normwise, mixed and
componentwise condition numbers of $\rt$-Sylvester
equation, which can be used to effectively estimate error
bounds. Moreover, we introduce the componentwise backward error for \eqref{sec2:mat eqqq} and derive a sharp and easily computable upper bound. %We also define the componentwise backward error for $\rt$-Sylvester matrix equation \eqref{sec2:mat eqqq} and the easier computable upper bound for the componentwise backward error is derived.
In the following we will introduce the definition of the directional  derivative, which will be used in the SCE algorithm.  For a function~$f:\R^n\mapsto \R^m$, the directional  derivative of $f$ at ~$x$ along the direction~$y\in \mathbb{R}^n$ is defined by
 $$
{\bf D}f(x;y)=\lim_{h\rightarrow 0}\frac{f(x+hy)-f(x)}{h}.
 $$
We now introduce the following map£º
\begin{equation} \label{Sylvester_map_Phi}
 \Phi: [A,\, B,\, C] \longmapsto X,
\end{equation}
where~$X$ is the unique solution of ~\eqref{sec2:mat eq}. The following lemma gives the explicit expression of the directional derivative .
\begin{lemma} \label{frechetExpression}
The map~$\Phi$ defined by \eqref{Sylvester_map_Phi} is continuous and directional differential at $
v=(\vect(A)^\rt,\vect(B)^\rt,\vect(C)^\rt)^\rt,
$
 and the directional derivative  of $\Phi$ at $v$ along the direction $[E,F,G]$ is the solution of the following $\rt$-Sylvester equation
 \begin{equation}\label{eq:Y}
 AY+Y^\rt B^\rt=G-EX-X^\rt F^\rt.
 \end{equation}
\end{lemma}
\begin{proof}
 Let $\delta >0$ and $E,\, F$ and $G$ be given, suppose $X+\delta Y$ is the exact solution of the following $\rt$-Slyvester equation
\begin{equation}\label{eq:peturbed TSy}
 (A+\delta E)(X+\delta Y)+(X+\delta Y)^\rt(B+\delta F)^\rt=C+\delta G.
\end{equation}
Subtracting from the unperturbed $\rt$-Sylvester equation~\eqref{sec2:mat eq}, forcing  $\delta \rightarrow 0$, and  using the corresponding directional derivative, we then proves the lemma. \qed
\end{proof}

This paper is organized as follows.  We conduct the componentwise backward error analysis in Section~\ref{sec:comp back}. In Section~\ref{sec:SCE}, the SCE-base condition estimation algorithms are proposed. Sections~4 and 5 contain the numerical examples and the concluding remarks.

\section{Componentwise Backward Error Analysis}\label{sec:comp back}

In this section, we introduce the componentwise backward error for $\rt$-Sylvester  equation \eqref{sec2:mat eq}, and derive the corresponding sharp and  computational bounds. %The similar results for $\star$-Sylvester  equation \eqref{sec1:mat1 eq} can be obtained.

\begin{definition}\label{def:Y}
Suppose~$Y$ is the computed solution of the $\rt$-Sylvester  equation \eqref{sec2:mat eq}, we define the componentwise backward error as
\begin{align*}
 \mu(Y)=\min ~\{~\epsilon :& (A+\Delta A)Y+Y^\rt (B+\Delta B)^\rt =C+\Delta C
 \\&|\Delta A|\leq \epsilon|A|,\,|\Delta B|\leq \epsilon|B|,\,
|\Delta C|\leq \epsilon|C|\},
\end{align*}
where $|\Delta A|\leq \epsilon |A|$, interpreted componentwise with $|A|=(|a_{ij}|)$.
\end{definition}
The following transformation removes the absolute values from the constrains in Definition \ref{def:Y} and replaces inequalities by equalities.  Let
$$
\vect(\Delta A)=D_1 \nu_1,\, \vect(\Delta B)=D_2 \nu_2,\, \vect(\Delta C)=D_3 \nu_3,
$$
where $D_1={\diag}(\vect(A))$, $D_2={\diag}(\vect(B))$, $D_3={\diag}(\vect(C))$. The the smallest value of $\epsilon$ satisfying $|\Delta A|\leq \epsilon|A|,\,|\Delta B|\leq \epsilon|B|$ and $|\Delta C|\leq \epsilon|C|$ is $\epsilon=\max\{\|\nu_1\|_\infty, \, \|\nu_2\|_\infty,\, \|\nu_3\|_\infty\}$, and so
\begin{eqnarray*}
 \mu(Y)&=&\min ~\Big \{~\left\|\begin{bmatrix}
   \nu_1\cr\nu_2\cr\nu_3
 \end{bmatrix}\right\|_\infty:  (A+\Delta A)Y+Y^\rt (B+\Delta B)^\rt =C+\Delta C,
 \\ &  & \vect(\Delta A)=D_1 \nu_1,\, \vect(\Delta B)=D_2 \nu_2,\, \vect(\Delta C)=D_3 \nu_3\Big\},
\end{eqnarray*}
In general, this equality constrained nonlinear optimization problem has no closed form solution. In the following theorem, we  give a sharp and easy-to-compute bound for $\mu(Y)$.
\begin{theorem}\label{t}
Let~$Y$ and  ~$\mu(Y)$ be defined as in Definition \ref{def:Y} and
$$
H=[~(Y^\rt\otimes I)D_1,(I\otimes Y^\rt)\Pi D_2,-D_3~],\quad \widetilde{R}=C-AY-Y^\rt B^\rt .
$$
Assume $H$ has full rank, let $r=\vect(\widetilde{R})$ and consider the QR decomposition
$$
H^\rt =Q\begin{bmatrix}R\cr 0\end{bmatrix},
$$
we have
$$
\mu(Y)\leq \overline{\mu} (Y):= \left\| Q ~\left[\begin{matrix}
\overline{z}_1\cr 0
\end{matrix} \right]\right\|_\infty \leq \sqrt{3}n\mu(Y),
$$
where  $\bar z_1=R^{-\rt} r$.
When $H$ is rank-deficient, $\mu(Y)$ is defined being infinite.
\end{theorem}

\begin{proof}
  Putting $\Delta A,\, \Delta B$ and $\Delta C$ to the right hand side of the perturbed equation
  $$
  (A+\Delta A)Y+Y^\rt (B+\Delta B)^\rt =C+\Delta C,
  $$
we have
\begin{equation} \label{b}
 \Delta A Y + Y^\rt \Delta B^\rt -\Delta C=\widetilde{R}.
\end{equation}
Applying the $\vect$ operation, \eqref{b} has the following form:
\begin{equation} \label{c}
(Y^\rt \otimes I)\vect(\Delta A)+(I\otimes Y^\rt)\Pi \vect(\Delta B)-\vect(C)=\vect(\widetilde{R}).
\end{equation}
Then the above equation can be written as the following linear system
\begin{equation}
 [~Y^\rt\otimes I,(I\otimes Y^\rt)\Pi,-I\otimes I~]
\left[ \begin{matrix} \vect(\Delta A)\cr \vect(\Delta B)\cr \vect(\Delta C)
\end{matrix}\right]=\vect(\widetilde{R}).
\end{equation}

Recalling the diagonal matrices $D_1=\diag(\vect(A)),\, D_2=\diag(\vect(B))$ and $D_3=\diag(\vect(C))$, we have
\begin{equation}\label{d}
[~Y^\rt\otimes I,(I\otimes Y^\rt)\Pi,-I\otimes I~]\left[\begin{matrix}D_1\cr & D_2\cr &&D_3\end{matrix}\right]
 \left[\begin{matrix}\nu_1\cr \nu_2\cr \nu_3 \end{matrix}\right]
 =\vect(\widetilde{R}).
\end{equation}
This is an underdetermined system of the form $Hz=r$, with $H\in \R^{n^2\times 3n^2}$ and $r=\vect(\widetilde{R})$. We seek the solution of minimal $\infty$-norm at $\mu(Y)$.

If $H$ is rank-deficient, then there may be no solution to $Hz=r$, in which case the componentwise backward error $\mu(Y)$ may be regarded as infinite. Assume, therefore, that $H$ is full rank. Using the QR factorization of $H^\rt$, then $Hz=r$ may be written as
$$
r=[R^\rt~ 0]Q^\rt z=[R^\rt~ 0]\begin{bmatrix}
  \bar z_1\cr \bar z_2
\end{bmatrix}=R^\rt \bar z_1.
$$
Thus $\bar z_1=R^{-\rt}r$ is uniquely determined and
$$
z=Q\begin{bmatrix}\bar z_1\cr \bar z_2\end{bmatrix}.
$$
Choosing $\bar z_2$ to minimize $\|z\|_\infty$ is equivalent to solving an overdetermined linear system in the $\infty$-norm sense, for which several methods are available.

We can obtain approximation to  the desired $\infty$-norm minimum by minimizing in the $2$-norm, which amounts to setting $\bar z_2=0$ (which yields $z={H^\dagger}^\rt r$, where $H^\dagger$ is the pseudo-inverse of $H$). In view of the fact that $s^{-1/2}\|t\|_2\leq \|t\|_\infty\leq \|t\|_2$ for $t\in \R^{s}$, it follows that
\begin{equation*}
  \mu(Y)\leq \overline \mu(Y) \leq \sqrt{3n}\mu(Y).
\end{equation*}\qed
\end{proof}

%{\tt say some words on the previous work}

\section{Small-Sample Condition Estimations}\label{sec:SCE}

In this section, based on a small-sample statistical condition
estimation method, we present a practical method for estimating
the condition numbers for the $\star$-Sylveter equations.
The small-sample statistical condition estimation (SCE) is proposed
by Kenny and Laub~\cite{KenneyLaub_SISC94}. It is an efficient
method for estimating the condition numbers for linear systems~\cite{KenneyLaubReese1998LS}, linear
least squares problems~\cite{KenneyLaubReese1998Linear}, eigenvalue problems~\cite{GudmundssonKenneyLaub1997Eig}, and roots of
polynomials~\cite{LaubXia2008Root}.  Based on the adjoint method and SCE, Cao and Petzold \cite{CaoPetzold2004ANM}  proposed an efficient method for estimating the error in the solution of the Sylvester matrix equations. Diao {\it et al}.  applied the SCE to the Sylvester equations~\cite{DiaoShiWei2013,DiaoXiangWei2012NLAA}, algebraic Riccati equations~\cite{DiaoLiuQiao2016AMC} and the structured Tikhonov regularization problem~\cite{DiaoWeiQiao2016JCAM}.

\subsection{Review on the SCE}\label{sec:review SCE}

We next briefly describe the SCE method. Given a differentiable function
$f:\R^{p}\rightarrow \R$, we are interested in its sensitivity
at some input vector $x$. From its Taylor expansion, we have
$$
f(x+ \delta d)-f(x)= \delta (\nabla f(x))^T d+O(\delta ^2),
$$
for a small scalar $\delta$, where
\[
\nabla f(x)= \left[ \frac{\partial f(x)}{\partial x_1},
\frac{\partial f(x)}{\partial x_2},\ldots,
\frac{\partial f(x)}{\partial x_p} \right]^T
\]
is the gradient of $f$ at $x$.
Then the local sensitivity, up to the first order
in $\delta$, can be measured by
$\|\nabla f(x)\|_2$. The condition number of $f$ at $x$ is asymptotically
determined by the norm of the gradient $\nabla f(x)$. It is shown in \cite{KenneyLaub_SISC94}
that if we select $d$ uniformly and randomly from the unit $p$-sphere
$S_{p-1}$ (denoted $U(S_{p-1})$), then the expected value
${\bf E}(|(\nabla f(x))^T d|/\omega_p)$ is $\|\nabla f(x)\|_2$,
where $\omega_p$ is the Wallis factor,
dependent only on $p$:
$$
\omega_p=
\begin{cases}
1, & \text{for}~p\equiv 1,\\
\frac{2}{\pi}, &\text{for}~p\equiv 2,\\
\frac{1 \cdot 3 \cdot 5 \cdots (p-2)}{2 \cdot 4 \cdot 6 \cdots (p-1)}, &
\text{for}~p ~ \text{odd} ~\text{and} ~p>2, \\
\frac{2}{\pi} \frac{2 \cdot 4 \cdot 6 \cdots (p-2)}
{1 \cdot 3 \cdot 5 \cdots (p-1)}, &
\text{for}~ p ~ \text{even}~ \text{and} ~p>2 ,
\end{cases}
$$
which can be accurately
approximated by
\begin{equation}\label{ss:p} \omega_p\approx
\sqrt{\frac{2}{\pi(p-\frac{1}{2})}}.
\end{equation}
Therefore,
$$
\nu=\frac{|(\nabla f(x))^T d|}{\omega_p}
$$
can be used to estimate $\|\nabla f(x)\|_2$, an approximation
of the condition number, with high probability. Specifically, for $\gamma >1$, we have
\begin{eqnarray*}
 \Prob\left(\frac{\|\nabla f(x)\|_2}{\gamma}\leq \nu \leq
\gamma \|\nabla f(x)\|_2\right)
\geq 1-\frac{2}{\pi \gamma}+O(\gamma^{-2}).
\end{eqnarray*}
Multiple samples $\{d_j\}$ can be used to increase the accuracy. The $k$-sample condition estimation
is given by
\begin{eqnarray*}
 \nu(k)
= \frac{\omega_k}{\omega_p}
\sqrt{|\nabla f(x)^T d_1|^2 +|\nabla f(x)^T d_2|^2 + \cdots +
|\nabla f(x)^T d_k|^2},
\end{eqnarray*}
where $d_1,d_2,...,d_k$ are orthonormalized after they
are selected uniformly and randomly from $U(S_{p-1})$.
For example,
\begin{align*}
\Prob\left(\frac{\|\nabla f(x)\|_2}{\gamma}\leq \nu(2) \leq \gamma
\|\nabla f(x)\|_2\right) &\approx 1-\frac{\pi}{4 \gamma^2},\cr
\Prob\left(\frac{\|\nabla f(x)\|_2}{\gamma}\leq \nu(3) \leq \gamma
\|\nabla f(x)\|_2\right) &\approx 1-\frac{32}{3\pi^2 \gamma^3},\cr
\Prob\left(\frac{\|\nabla f(x)\|_2}{\gamma}\leq \nu(4) \leq \gamma
\|\nabla f(x)\|_2\right) &\approx 1-\frac{81\pi^2}{512 \gamma^4}.
\end{align*}
If we choose~$k=3,\, \gamma=10$, then $\nu(3)$ has a probability $0.9989$ to be within an order of $\| \nabla f(x) \|_2$ (i.e., between $\| \nabla f(x) \|_2/10$ and $10 \cdot \| \nabla f(x) \|_2$).

First we introduce {\sf unvec} operator, for ~$v=(v_1,v_2,\ldots,v_{n^2})^\rt \in \R^{n^2}$, then $A=$ {\sf unvec} $(v)$ with~$a_{ij}=v_{i+(j-1)n}$. These results can be readily generalized
to vector-valued or matrix-valued functions by viewing $f$ as
a map from $\R^s$ to $\R^t$, possibly after the operations
$\vect$ and $\sf unvec$ to transform data between matrices and vectors,
where each of the $t$ entries of $f$ is a scalar-valued function. The main  computational cost of the SCE is to evaluate the directional derivative of the given mapping $f$ at the input data $x$. Usually when we solve the  problem in numerical linear algebra by direct methods, we have some decompositions of the matrix which can be used to compute the directional derivative  efficiently. For $\rt$-Sylvester equations, the generalized Schur algorithm in \cite{ChiangChuLin2012AMC} had been proposed. We can utilize the generalized Schur algorithm to compute the directional derivative efficiently based on Lemma \ref{frechetExpression}. In practice computation, we do not have the exact solution $X$ but we can use the computed one to approximate the directional derivative.

%From~\ref{Sylvester_map_Phi} and Lemma \ref{frechetExpression}, %~\ref{frechetExpression}£¬
%$\Phi$~Fr\'{e}chet derivative
%$$
%D{\Phi}(v)\circ (E,F,G)=-P^{-1} \left[~ X^\rt\otimes I,~ (I\otimes
%X^\rt)\Pi,~ -I ~\right]\left[\vect(E)^\rt, \vect(F)^\rt,
%\vect(G)^\rt\right]^\rt,
%$$
%where $E,F,G \in \R^{n \times n}$. In practice computation, we do not have the exact solution $X$ but we can use the computed solution to approximate the directional derivative.

\subsection{Normwise perturbation analysis}

In this section,we will apply the SCE technique to the $\rt$-Sylvester equation \eqref{sec2:mat eqqq} under normwise perturbations. For the solution of $X$,
we are interested in the condition estimation at the point $[A, \, B,\, C]$.
Let $[A\, \,B,\, C]$ be perturbed to $[A+\delta E\, \, B+\delta F,\, C+\delta G]$ in \eqref{eq:peturbed TSy}, where $\delta \in \R$ and
$E, F, G \in \R^{n\times n}$  with $\|[ E,  F, G]\|_F = 1$. According to Section~\ref{sec:review SCE},
we first need to evaluate the directional derivative
${\bf D} \Phi([A,\, B,\, C];[E,\, F,\, G])$ and for that
from Lemma~\ref{frechetExpression}, we need to solve the $\rt$-Sylvester equation \eqref{eq:Y}. When we use Algorithm ${\bf \mathrm { TSylvester_R}}$ in \cite{ChiangChuLin2012AMC} to solve \eqref{sec2:mat eq}, the generalized real Schur decomposition of the pencil $A-\lambda B^\rt$ is already available, thus \eqref{eq:Y} can be solved without minimal additional costs.  We are now ready to use the SCE techniques in Section~\ref{sec:review SCE} to estimate the
condition of the $\rt$-Sylvester equation \eqref{sec2:mat eqqq} under normwise perturbations.
Algorithm~\ref{al:norm} computes a relative condition estimation matrix for the solution $X$ of $\rt$-Sylvester equation \eqref{sec2:mat eqqq}. Inputs
to the method are the matrices $A,\, B \in  \Rnn$, $C\in \Rnn$ and the computed solution $X$. The output is
a relative condition estimation matrix $R_{\rm rel}^{\rt-\mathrm{SYL},(k)}$ and a relative normwise condition number estimation $\kappa_{F, SCE}^{\rt-\mathrm{SYL,(k)}}$ for $\kappa^\mathrm{T-SYL} $ which is defined in \eqref{eq:cond definition}.
Again, the method requires the generalized real Schur decomposition of $A-\lambda B^\rt$, which is generally available after solving the $\rt$-Sylvester equation. The integer $k\geq 1$ refers to
the number of samples of perturbations to the input data.
When $k = 1$, there is obviously no need to orthonormalize the set of
vectors in Step~1 of the algorithm.

{\small
\begin{algorithm}\label{al:norm}
\caption{Subspace condition estimation for the solution $X$ of
$\rt$-Sylvester equation~(\ref{sec2:mat eq}) under normwise perturbation}

\begin{itemize}
\item[1.] Generate pairs $(E_1,F_1,G_1), (E_2,F_2,G_2),\ldots, (E_k,F_k,G_k)$ with
entry is in ${\cal N}(0,1)$. Use the Modified Gram-Schmidt (MGS) orthogonalization process for
$$
\left[\begin{matrix}\vect(E_1)&\cdots& \vect(E_k)\cr
\vect(F_1)&\cdots& \vect(F_k)\cr \vect(G_1)&\cdots& \vect(G_k)
\end{matrix}\right],
 $$
to obtain an orthonormal matrix $[q_1,q_2,\ldots,q_k]$. Convert $q_i$ to $(\widetilde{E_i},\, \widetilde{F_i}, \widetilde{G_i})$ with the {\sf
unvec} operation.
\item[2.] Approximate $\omega_p$ and $\omega_k$ by (\ref{ss:p}), with $p={3n^2}$.
\item[3.] Solve the following $\rt$-Sylvester equation via the generalized Schur algorithm \cite{ChiangChuLin2012AMC}:
$$
AY_i+Y_i^\rt B^\rt = \widetilde{G}_i-\widetilde{E}_i X+X^\rt \widetilde{F}_i^\rt.
$$
\item[4.] Calculate respectively the absolute condition matrix and the normwise absolution condition estimation
$$
K_{\rm
abs}^{\rt-\mathrm{SYL},(k)}=\frac{\omega_k}{\omega_p}\sqrt{|Y_1|^2+
|Y_2|^2+\cdots+|Y_k|^2}, \ \ \
n_{F,
SCE}^{\rt-\mathrm{SYL,(k)}}%=\frac{\omega_k}{\omega_p}\sqrt{||Y_1||_F^2+||Y_2||_F^2+\cdots+||Y_k||_F^2}
=\|K_{\rm
abs}^{\mathrm{SYL},(k)}\|_F,
$$
where the square root is taken for each elements of the matrix. Let the relative condition matrix $R_{\rm rel}^{\rt-\mathrm{SYL},(k)}$ be the matrix $\|[A,~B,~C]\|_F\cdot K_{\rm abs}^{\rt-\mathrm{SYL},(k)}$ divided componentwise by $X$, leaving
entries of $K_{\rm abs}^{\rt-\mathrm{SYL},(k)}$ corresponding to zero entries of $X$ unchanged.
Compute  $\kappa_{F, SCE}^{\rt-\mathrm{SYL,(k)}}=n_{F,
SCE}^{\rt-\mathrm{SYL,(k)}}/||X||_F$.
\end{itemize}
\end{algorithm}
}

In Table~\ref{ta:1}, we report the flop counts of Algorithm \ref{al:norm}.
\begin{table}[!htbp]\label{ta:1}
\caption{\label{tab:1}Computational complexity of Algorithm \ref{al:norm}. }
\begin{center}
\begin{tabular}{cc}
\hline
 Step & Flops  \\
\hline
  1& $\Oh(6 k^2 n^2 )$ \cr
  2& $\Oh(1)$ \cr
  3& $\Oh(k n^3)$\cr
  4& $\Oh(3kn^2)$\cr
 \hline
\end{tabular}
\end{center}
\end{table}
We can see that the total flop counts of Algorithm \ref{al:norm} are $\Oh(k n^3)$, which are the same order of the flop counts of the generalized Schur algorithm \cite{ChiangChuLin2012AMC} for solving the $\rt$-Sylvester equation \eqref{sec2:mat eq}. It is generally true that solving a problem and estimating its condition involve a similar amount of work, indicating comparable levels of difficulty.

\subsection{Componentwise perturbation analysis}

Componentwise perturbations are relative to the magnitudes of
the corresponding entries in the input arguments, where the perturbation $\Delta A$ satisfies $|\Delta A| \leq \epsilon |A|$. These perturbations
may arise from input error or from rounding error, and hence are
the most common perturbations encountered in practice. In fact,
most of error bounds in LAPACK are considered componentwise
\cite[Section 4.3.2]{Anderson1999Lapack}.
We often want to find the condition of a function with respect to
componentwise perturbations. For $\Phi$ defined in \eqref{Sylvester_map_Phi}, SCE is flexible enough to accurately gauge the sensitivity of
matrix functions subject to componentwise perturbations.
Define the linear mask function
$$
h([E,~F,~G])=[E,~F,~G]\odot [A,~B,~C],\quad E\in \R^{n\times n},\,F\in \R^{n\times n},\, G\in \R^{n\times n},
$$
where $\odot$ denotes the Hadamard componentwise multiplication.  When ${\cal E} \in \R^{n\times 3n}$ is
the matrix of all ones, then $h({\cal E})=[A,~B,~C]$ and
$$
h({\cal E}+[E,~F,~G])=[A,~B,~C]+h([E,~F,~G]).
$$
Thus $h([E,~F,~G])$ is a componentwise perturbation on $[A,~B,~C]$,
and $h$ converts a general perturbation $\cal E$ into
componentwise perturbations on $[A,~B,~C]$. Therefore, to obtain
the sensitivity of the solution with respect
to relative perturbations, we simply evaluate the Fr\'{e}chet derivative of
$$
\Phi([A,~B,~C])=\Phi(h({\cal E}))
$$
with respect to $\cal E$ in the direction $[\Delta A,~\Delta B,~\Delta C]$, which is
\begin{align*}
&{\bf D}(\Phi \circ h)\left({\cal E};[E,~F,~G])\right)=
{\bf D}\Phi (h({\cal E})){\bf D}h\left({\cal E};[E,~F,~G]\right)\\
&=
{\bf D}\Phi ([A,~B,~C])h\left([E,~F,~G]\right)=
{\bf D}\Phi \left([A,~B,~C]);h\left([E,~F,~G]\right)\right),
\end{align*}
since $h$ is linear. Thus, to estimate the condition of $X$ for componentwise perturbations, we first generate the perturbations  $E$, $F$ and $G$ and multiply them componentwise by $A$, $B$ and $C$,
respectively. The remaining steps are the same as
the corresponding steps in Algorithm~\ref{al:norm},
as shown in Algorithm~\ref{algo:subcomp}.

{\small
\begin{algorithm}\label{algo:subcomp}
\caption{Subspace condition estimation for the solution $X$ of
$\rt$-Sylvester equation~(\ref{sec2:mat eq}) under componentwise perturbation}

\begin{itemize}
\item[1.] Generate pairs $(E_1,F_1,G_1), (E_2,F_2,G_2),\ldots, (E_k,F_k,G_k)$ with
entry is in ${\cal N}(0,1)$. Use the Modified Gram-Schmidt (MGS) orthogonalization process for
$$
\left[\begin{matrix}\vect(E_1)&\cdots& \vect(E_k)\cr
\vect(F_1)&\cdots& \vect(F_k)\cr \vect(G_1)&\cdots& \vect(G_k)
\end{matrix}\right],
 $$
to obtain an orthonormal matrix $[q_1,q_2,\ldots,q_k]$. Convert $q_i$ to $(\widetilde{E_i},\, \widetilde{F_i}, \widetilde{G_i})$ with the {\sf
unvec} operation. Let $\widetilde{E_i^c}=\widetilde{E_i}\odot A,\, \widetilde{F_i^c}=\widetilde{F_i}\odot B$ and $\widetilde{G_i^c}=\widetilde{G_i}\odot C$.
\item[2.] Approximate $\omega_p$ and $\omega_k$
by~(\ref{ss:p}), with $p={3n^2}$.
\item[3.] Solve the following $\rt$-Sylvester equation via the generalized Schur algorithm \cite{ChiangChuLin2012AMC}:
$$
AY_i+Y_i^\rt B^\rt = \widetilde{G^c_i}-\widetilde{E_i^c} X+X^\rt (\widetilde{F^c_i})^\rt.
$$
\item[4.] Calculate the absolute condition matrix
$$
M_{\rm
abs}^{\rt-\mathrm{SYL},(k)}=\frac{\omega_k}{\omega_p}\sqrt{|Y_1|^2+|Y_2|^2+\cdots+|Y_k|^2}.
$$
Compute the relative componentwise condition matrix $C_{\rm rel}^{\rt-\mathrm{SYL},(k)} = M_{\rm abs}^{\rt-\mathrm{SYL},(k)}/X$ (division  carried out componentwise), leaving
entries corresponding to zero entries of $X$ unchanged.
Compute
\begin{equation}\label{eq:es mixed}
m^{\mathrm{T-SYL},(k)}_\mathrm{SCE} := \frac{\|M^{\mathrm{T-SYL},(k)}\|_{\max}}{\|X\|_{\rm
max}} ,
\quad c^{\mathrm{T-SYL},(k)}_\mathrm{SCE} :=
\left\|\frac{M^{\mathrm{T-SYL},(k)}}{X}\right\|_{\rm max},
\end{equation}
where $\|X\|_{\rm
max}=\max_{ij}|X_{ij}|.$
\end{itemize}
\end{algorithm}
}

\section{Numerical Examples}

In this section, we demonstrate our test results of some
numerical examples to illustrate componentwise backward errors and condition estimations presented earlier. Numerical experiments were carried out on a machine with Intel i5  4590 @3.3GHz CPU, 8G RAM and 1TB hard driver running Windows 7 professional, using \textsc{Matlab} 8.5 with a machine precision
$\varepsilon=2.2 \times 10^{-16}$.

%We use Gaussian elimination with partial pivoting to solve \eqref{sec2:mat eqqq}. The computed solution is denoted by $\tilde X$ and recall
%$\Delta X=\tilde X-X$.

We generate the perturbation matrices as follws
\begin{equation}\label{eq:perturbation gene}
\Delta A=\epsilon\cdot \Delta_A\odot A,\quad, \Delta B=\epsilon\cdot \Delta_B\odot B,\quad, \Delta C=\epsilon\cdot \Delta_C\odot C,
\end{equation}
where $\Delta_A,\,\Delta_B$ and $\Delta_C$ are random matrices with each entries being uniformly distritbuted in $(-1,\,1)$. Let  $\tilde X$ be the solution of \eqref{eq:perturbed T-Syl}. We use Gaussian elimination with partial pivoting to solve \eqref{eq:perturbed T-Syl} in Kroncker product form. Recall that $\Delta X=\tilde X-X$. Let us denote the true relative errors
$$
\gamma_\kappa=\frac{\|\Delta X\|_F}{\|X\|_F},\, \gamma_m=\frac{\|\Delta X\|_{\max}}{\|X\|_{\max }}, \, \gamma_c=\left\|\frac{\Delta X}{X}\right\|_{\max}.
$$
Clearly, from the definitions of condition numbers in \eqref{eq:cond definition}, we have the following inequalities between the first order perturbation bounds and the corresponding exact relative errors:
$$
\gamma_\kappa \leq \kappa^\mathrm{T-SYL} \epsilon,\, \gamma_m \leq m^\mathrm{T-SYL} \epsilon,\gamma_c \leq c^\mathrm{T-SYL} \epsilon.
$$
Also, from Algorithms \ref{al:norm} and \ref{algo:subcomp},  we can compute the condition estimates $\kappa_{F, SCE}^{\rt-\mathrm{SYL,(k)}},\, m^{\mathrm{T-SYL},(k)}_\mathrm{SCE} $ and $c^{\mathrm{T-SYL},(k)}_\mathrm{SCE}$ which can be used to approximate the posterior perturbation bounds for \eqref{sec2:mat eq}.

\begin{example}\label{ex:small}
This example is quoted from \cite[Example 3.3]{ChiangChuLin2012AMC}. We use \textsc{Matlab} $randn(n)$ to compute an $n \times n$ random matrix with entries being normal distributed. Let $n=2$, $Q\in \R^{n\times n}$ be orthogonal, the exact solution be $X_e$, where
\begin{align*}
X&=Q^\rt \begin{bmatrix}
  10^{-m} &0\cr 0 & 10^m
\end{bmatrix} Q,\, A=\begin{bmatrix}
  randn(1)&0\cr
  randn(1)& 10^{-m}
\end{bmatrix}Q,\\
 B&=\begin{bmatrix}
  randn(1)&0 \cr randn(1)& 2\cdot 10^{-m}
\end{bmatrix}Q,
\end{align*}
and $C=AX+X^\rt B^\rt$.  For different $m$ and $\epsilon$, we compare the true relative errors with the true and estimated first order perturbation bounds in Table \ref{Ta:1} and \ref{Ta:2}. For Algorithms~\ref{al:norm} and \ref{algo:subcomp}, we choose $k=3$. Typically the condition estimates fall reliably within the factors between a tenth and ten folds of the true condition numbers \cite[Chap. 15]{Higham2002Book}. From Tables~\ref{Ta:1} and \ref{Ta:2}, it is easy to see that the condition of \eqref{sec2:mat eq} worsens as $m$ increases. The first order perturbation bounds approximate the true relative bounds well. On the other hands, the  SCE-base condition estimates underestimate the true relative error within the factor 1/10, which is consistent with the theory of SCE.

%\begin{table}\centering
%\begin{tabular}{||ccccccccc||}
%\hline
%%$7.3346\cdot 10^{-15}$& $6.0461\cdot 10^{-13}$&$3.8301\cdot 10^{-13}$&$7.4461\cdot 10^{-15}$& $7.5221\cdot 10^{-14}$ &$6.0673\cdot 10^{-15}$ &$2.5400\cdot 10^{-14} $&$2.5659\cdot 10^{-13}$ &$2.0697\cdot 10^{-14}$\\
%%\hline
%7.3346e-15 &6.0461e-13&3.8301e-13& 7.4461e-15&7.5221e-14&6.0673e-15 &2.5400e-14 &2.5659e-13 &2.0697e-14\\
%\hline
%1.0575e-13& 1.1191e-11 & 8.1177e-12 &8.2668e-14 &5.8057e-13 &6.5809e-14 &2.8237e-13 &1.6209e-12 &2.0623e-13\\
%\hline
%7.2636e-12 & 4.7120e-09 &2.9189e-09 & 7.3287e-12 &9.8187e-10 &5.7456e-11 &7.3289e-12 &2.2592e-09 &1.3241e-10\\
%\hline
%5.2194e-10 & 3.0334e-08 &1.1757e-08 &5.4127e-10 &5.4345e-09 &4.6896e-10 &5.4127e-10 &5.4345e-09 &4.6896e-10\\
%\hline
%9.4721e-08 & 4.7188e-06 &2.8060e-06 &9.8546e-08 &1.3253e-06 &1.4791e-07 &9.8546e-08 &2.4434e-06 &2.8748e-07\\
%\hline
%\end{tabular}
%\end{table}

\begin{table}
\caption{\label{Ta:1}  Comparing the true normwise perturbation bounds with the first order normwise bounds}%Compare the true normwise perturbation bounds with the first order normwise bounds}
\centering
\begin{tabular}{||ccccc||}
\hline
%$7.3346\cdot 10^{-15}$& $6.0461\cdot 10^{-13}$&$3.8301\cdot 10^{-13}$&$7.4461\cdot 10^{-15}$& $7.5221\cdot 10^{-14}$ &$6.0673\cdot 10^{-15}$ &$2.5400\cdot 10^{-14} $&$2.5659\cdot 10^{-13}$ &$2.0697\cdot 10^{-14}$\\
%\hline
$\epsilon$ & $m$ & $\gamma_\kappa$ & $\kappa^\mathrm{T-SYL} \cdot \epsilon$ &$\kappa_{F, SCE}^{\rt-\mathrm{SYL,(k)}}\cdot \epsilon$\\
\hline
$10^{-8}$& 2 &9.5396e-08& 3.4269e-06&2.8820e-06 \\
\hline
&4 &2.4689e-04& 3.0980e-03&8.7881e-04\\
\hline
&6& 9.7546e-04& 5.5240e-02&3.0854e-02\\
\hline
&8&3.5031e-01& 4.9295e+00&3.5165e+00\\
\hline
&10&1.0388e+00& 6.3475e+02&4.6246e+02\\
\hline
$10^{-16}$& & & & \\
\hline
&2&1.8310e-17&1.1567e-14&5.7126e-15\\
\hline
&4&4.8182e-14&3.4386e-12&1.1154e-12\\
\hline
 &6&1.8526e-12& 4.5175e-10&2.8984e-10\\
 \hline
&8&6.2313e-09& 1.5478e-07&1.1568e-07\\
\hline
&10&3.7921e-07&4.1497e-05&1.6493e-05\\
\hline
\end{tabular}
\end{table}

\begin{table}\centering
\caption{\label{Ta:2}  Comparing the true mixed, componentwise perturbation bounds with the first order mixed, componetnwise bounds}
{\small
\begin{tabular}{||cccccccc||}
\hline
$\epsilon$ & $m$ & $\gamma_m$ & $m^\mathrm{T-SYL} \cdot \epsilon$ &$m^{\mathrm{T-SYL},(k)}_\mathrm{SCE}\cdot \epsilon$& $\gamma_c$ & $c^\mathrm{T-SYL}\cdot \epsilon$ & $c^{\mathrm{T-SYL},(k)}_\mathrm{SCE}\cdot \epsilon$\\
\hline
$10^{-8}$& 2 &1.0653e-07&6.7982e-07&7.0811e-08&1.0653e-07&6.9993e-07&7.3014e-08\\
\hline
&4 &2.8197e-04&6.1050e-04&3.1166e-05&3.5893e-04&7.7713e-04&3.9672e-05\\
\hline
&6&1.1605e-03&3.7938e-02&1.7024e-03&1.1605e-03&3.7938e-02&1.7024e-03\\
\hline
&8&3.4915e-01&1.3625e+00&9.0707e-02&3.5217e-01&1.3625e+00&9.0707e-02\\
\hline
&10&1.0052e+00&1.7985e+02&2.0008e+01&1.0912e+00&1.7985e+02&2.0008e+01\\
\hline
$10^{-16}$& & & & & & & \\
\hline
 &2& 1.7901e-17&2.4946e-16&1.9630e-17&5.7210e-16&7.3778e-16&8.5091e-17\\
\hline
& 4 &4.7751e-14&3.0209e-13&2.7450e-14&5.4371e-13&3.4415e-12&3.1275e-13\\
\hline
 &6 &1.5149e-12&2.4979e-11&2.3467e-12&1.7239e-11&2.8425e-10&2.6705e-11\\
 \hline
&8  &5.9587e-09&1.0139e-08&1.1128e-09&6.7807e-08&1.1537e-07&1.2663e-08\\
\hline
&10&  3.3359e-07&2.7142e-06&1.3414e-07&2.0868e-06&1.5021e-05&7.6833e-07\\
\hline
\end{tabular}
}
\end{table}

Algorithms \ref{al:norm} and \ref{algo:subcomp} output the condition matrix which bounds componentwise the true relative error of each entry of  $X$. Let us denote the overestimation matrices
$$
O^{N}=\frac{R_{\rm rel}^{\rt-\mathrm{SYL},(k)}\cdot \epsilon}{\Delta X/X},\quad
O^{C}=\frac{C_{\rm rel}^{\rt-\mathrm{SYL},(k)}\cdot \epsilon}{\Delta X/X},
$$
where $\epsilon$ denotes the perturbation magnitude in \eqref{eq:perturbation gene}, and $R_{\rm rel}^{\rt-\mathrm{SYL},(k)}$ and $C_{\rm rel}^{\rt-\mathrm{SYL},(k)}$ are outputs of Algorithms \ref{al:norm} and \ref{algo:subcomp}, respectively. We test 1000 samples of $(A,B, C)$ and plot the mean matrices $O^{N}$ and $O^{C}$ in Figure~\ref{fig:small}, for $k=3$ and $\epsilon=10^{-8}$. The X-axis of Figure~\ref{fig:small} denotes the index of  $\vect(X)$.  The graphs on the left and right of Figure~\ref{fig:small} display respectively the mean values of the overestimations given by Algorithms~\ref{al:norm} and \ref{algo:subcomp}. Clearly, Algorithm~\ref{algo:subcomp} gives better  estimates.

\begin{figure}
\centering
\includegraphics[width=6.4in]{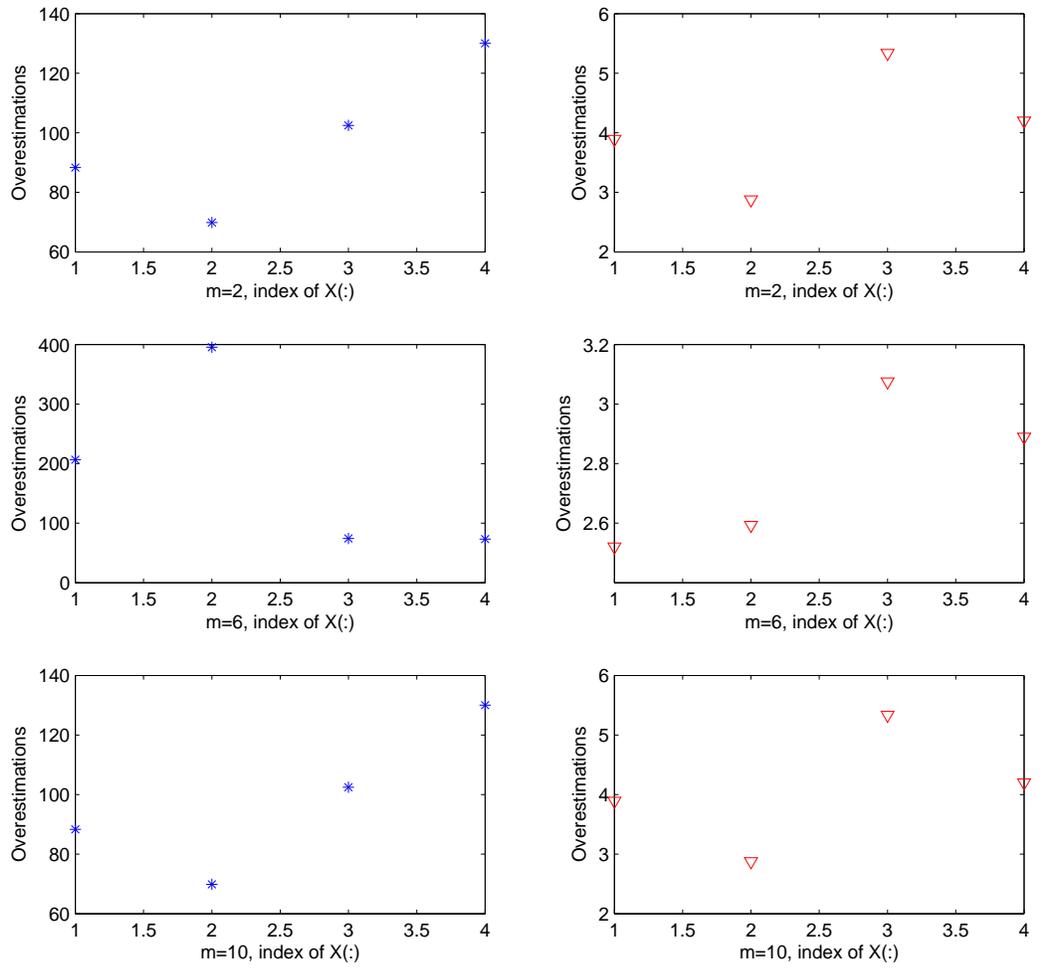}
\caption{Example 2. Overestimation of condition over 1000 samples}
\label{fig:small}
\end{figure}

\end{example}

%%%%%%%%%%%%%%%%%%%%%%%%%

\begin{example}
  This example came from \cite{ChiangChuLin2012AMC}. Let $\widehat  A,\, \widehat  B \in \R^{n\times n}$ be real lower-triangular matrices with given diagonal elements (denoted by $a,\, b \in \R^n$) and random strictly lower-triangular elements. They are the reshuffled by the orthogonal matrices $Q,\, Z \in \R^{n\times n}$ to form $(A,\, B)=(Q {\widehat  A} Z,\, Q {\widehat  B} Z)$. In \textsc{Matlab} commands, we have
  \begin{align*}
  \widehat A&=tril(randn(n),-1)+diag(a),\, \widehat  B=tril(randn(n),-1)+diag(b),\\
   X&=randn(n,n),
  \end{align*}
  and the right hands $C=AX+X^\top B^\top$.  We generate 1000 samples of $A,\,B$ and $X$ with $n=40$, and for each sample, the perturbations on $A,\, B$ and $C$ are generated as in the previous examples. We display the mean values of $O^{N}$ and $O^{C}$ over 1000 samples in Figure~\ref{fig:ex40}, for $k=3$ and $\epsilon = 10^{-16}$ in \eqref{eq:perturbation gene}. The X-axis of Figure~\ref{fig:ex40} denotes the index of  $\vect(X)$. From Figure~\ref{fig:ex40}, the componentwise condition estimation matrix $C_{\rm rel}^{\rt-\mathrm{SYL},(k)}$ gives reliable perturbation bounds. The mean value of entries of $O^{C}$ is 0.1991 and the variance is 1.9140. On the other hand, the mean value of entries of $O^{N}$ is 72.2192 and the variance is $2.3450\cdot 10^5$. So Algorithm~\ref{algo:subcomp} gives superior condition estimates.

\begin{figure}
\centering
\includegraphics[width=6.4in]{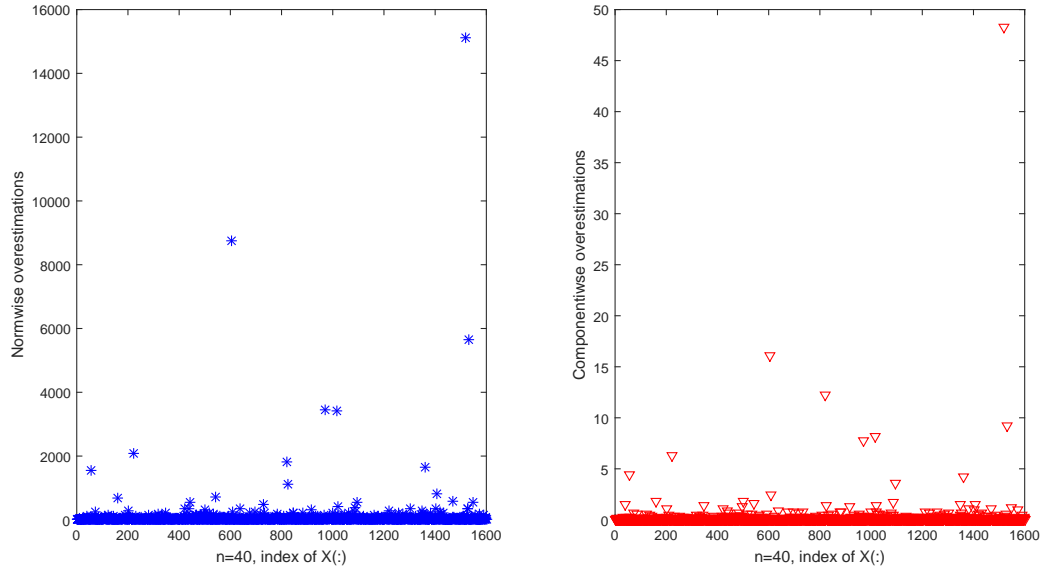}
\caption{Example 3. Overestimation of condition over 1000 samples}
\label{fig:ex40}
\end{figure}

\end{example}

\begin{example}
  In this example, we test the effectiveness of the proposed componentwise backward errors. The triples $A,\, B$ and $C$ are as in Example \ref{ex:small}. The perturbations $\Delta A,\, \Delta B$ and $\Delta C$ are generated as in \eqref{eq:perturbation gene}. Let $Y$ satisfies the perturbed $\top-$Sylvester equation
  \begin{equation*}
(A+\Delta A)Y+Y^\rt (B+\Delta B)^\rt=C+\Delta C,
\end{equation*}
which is solved in Kronecer production form by Gaussian elimination with partial pivoting. Denote
$$
\epsilon^*= \min ~\{~\epsilon : |\Delta A|\leq \epsilon|A|,\,|\Delta B|\leq \epsilon|B|,\,
|\Delta C|\leq \epsilon|C|\}.
$$
We vary the perturbation magnitudes $\epsilon$ in \eqref{eq:perturbation gene} from $10^{-3}$ to $10^{-9}$ and compute $\overline{\mu} (Y)$ in Theorem \ref{t} and  $\eta(Y)$ in \eqref{eq:eta} for different values of $m$. The results are displayed in Table~\ref{Ta:4}. When $m$ increases, the condition of the $\top$-Sylverter equation worsens, as indicated in Example~\ref{ex:small}. For most of cases, $\overline{\mu} (Y)$ has the same order as or one order higher than $\epsilon^*$. For $m=10$, when the perturbations $\epsilon$ are small, $\overline{\mu} (Y)$ seriously overestimates the true componentwise backward error. On the other hand, the normwise backward error $\eta(Y)$ does not estimate $\epsilon^*$ accurately even for well conditioned problem under small perturbations, as for $m=6$ and $\epsilon=10^{-6}$.

\begin{table}\centering
\caption{\label{Ta:4}  Comparing $\epsilon^*$, componentwise and normwise backward errors}
\begin{tabular}{||ccccc||}
\hline
$\epsilon$ & $m$ & $\epsilon^*$ & $\overline{\mu} (Y)$ &$\eta(Y)$\\
\hline
$10^{-3}$& & & &  \\
\hline
& 2&8.9398e-04& 1.4532e-04 & 8.9601e-03 \\
\hline
&4 &9.8239e-04 &7.2812e-04& 4.6979e-01\\
\hline
&6&9.5328e-04&7.0167e-04& 3.1161e+00 \\
\hline
&8&8.8183e-04&3.6163e-04& 1.4204e+00\\
\hline
&10&8.8183e-04&3.6163e-04& 1.4204e+00\\
\hline
$10^{-6}$& & & &  \\
\hline
 &2& 9.2596e-07 & 3.2334e-07 & 4.6975e-06 \\
\hline
& 4 &8.6877e-07 &5.1044e-07& 4.2311e-04\\
\hline
 &6 &9.8269e-07&8.6473e-07& 3.5344e-01\\
 \hline
&8  &9.9547e-07&7.2417e-07& 8.0498e-01\\
\hline
&10& 9.5721e-07& 2.6332e-07& 4.1066e+00\\
\hline
$10^{-9}$& & &  &\\
\hline
 &2& 9.8999e-10&7.9092e-10& 9.4315e-09\\
\hline
& 4 &9.2491e-10&7.4925e-10 &1.0222e-06\\
\hline
 &6 &8.8399e-10 &8.3860e-10& 1.6531e-04\\
 \hline
&8  &8.6229e-10&5.1742e-10& 5.5369e-03\\
\hline
&10& 9.6373e-10&8.5358e-09& 3.3068e-01\\
\hline
$10^{-12}$& & & &\\
\hline
 &2& -9.8837e-13 &8.5948e-13& 2.1765e-11\\
\hline
& 4 &8.6013e-13&4.6452e-13& 3.2712e-09\\
\hline
 &6 &8.1023e-13&4.6365e-12& 1.8620e-07\\
 \hline
&8  &7.7456e-13&1.6637e-09& 3.6200e-06\\
\hline
&10& 9.2831e-13&2.3684e-07& 1.7668e-05\\
\hline
\end{tabular}
\end{table}

\end{example}

\section{Concluding Remarks}

We have considered the condition and errors of $\star$-Sylvester equations under  componentwise perturbations. Backward errors have been defined and the small-sample condition estimation technique has been applied to estimate the condition of $\star$-Sylvester equations. Numerical experiments show our algorithm under componentwise perturbations produces accurate condition and error estimates which reflect true condition and errors accurately. Moreover, the new derived bound for the componentwise backward errors is sharp and reliable according to the numerical experiments for well-conditioned or moderate ill-conditioned problems under large or small perturbations. A possible future research topic is to apply the SCE to other type $\star$-Sylvester equation \cite{ChiangChuLin2012AMC}.

\section*{Acknowledgements}
H. Diao is partially supported by the
National Natural Science Foundation of China under grant 11001045.

\section*{References}

%\bibliography{diao160713}

\end{document}